\documentclass[10pt]{amsart}
\usepackage{geometry} 
\usepackage{graphicx}
\usepackage{amsmath}
\usepackage{amsthm}
\usepackage{amssymb}
\usepackage{color}
\usepackage{amscd}
\usepackage{amsfonts}
\usepackage{graphicx}
\usepackage{epsfig}
\usepackage{amssymb,latexsym}

\newtheorem{theorem}{Theorem}[section]

\newtheorem{definition}{Definition}[section]
\newtheorem{lemma}[theorem]{Lemma} 
\newtheorem{proposition}[theorem]{Proposition}

\makeatletter
\numberwithin{equation}{section}
\makeatother

\def\bx{\mathbf{x}}
\def\bX{\mathbf{X}}
\def\by{\mathbf{y}}

\def\bz{\mathbf{z}}

\def\bbg{\mathbf{g}}
\def\bbh{\mathbf{h}}
\def\bG{\mathbf{G}}

\def\R{\mathbb{R}}
\def\bn{\mathbf{n}}

\def\bcero{\mathbf{0}}
\def\buno{\mathbf{1}}

\title[Constrained optimization]{Constrained optimization through fixed point techniques}
\author{Pablo Pedregal}
\thanks{
Universidad de Castilla La Mancha, INEI, Ciudad Real (Spain). Research supported by
MTM2013-47053-P of the Mineco (Spain).  
 e-mail: {\tt pablo.pedregal@uclm.es
}}

\begin{document}

\maketitle
    \begin{abstract}
We introduce an alternative approach for constrained mathematical programming problems. It rests on two main  aspects: an efficient way to compute optimal solutions for unconstrained problems, and multipliers regarded as variables for a certain map. Contrary to typical dual strategies, optimal vectors of multipliers are sought as fixed points for that map. 
Two distinctive features of the procedure are worth highlighting: its simplicity and flexibility for the implementation, and its convergence properties.
    \end{abstract}

{\bf Key Words.} Unconstrained optimization, multipliers, optimality conditions.

\vspace{10pt} {\bf AMS(MOS) subject classifications.} 65K05,
54H25.
\section{Introduction}
We are concerned here with the general, standard mathematical program
\begin{equation}\label{probgenn}
\hbox{Minimize in }\bx\in\R^N:\quad f(\bx)\quad\hbox{subject to}\quad \bbh(\bx)=\bcero,\bbg(\bx)\le\bcero,
\end{equation}
for a smooth, real function $f:\R^N\to\R$, and smooth, vector-valued mappings $\bbh:\R^N\to\R^n$, $\bbg:\R^N\to\R^m$. Karush-Kuhn-Tucker (KKT) optimality conditions are among the basic techniques taught and learnt in optimization courses (there are hundreds of textbooks on the subject, see for instance \cite{hiriarturruty}).
They involved, in addition to $\bx$ itself, multipliers $\bz\in\R^n$, $\by\in\R^m$, for all of the constraints to be respected in the problem. Under appropriate constraint qualifications, that are not part of our discussion here, local solutions of (\ref{probgenn}) are to be found among the triplets $(\bx, \by, \bz)\in\R^N\times\R^m\times\R^n$ complying with
\begin{gather}
\nabla f(\bx)+\bz\cdot\nabla\bbh(\bx)+\by\cdot\nabla\bbg(\bx)=\bcero,\label{optim}\\
\bbh(\bx)=\bcero,\quad \by\cdot\bbg(\bx)=0,\nonumber\\
\by\ge\bcero,\quad \bbg(\bx)\le\bcero.\nonumber
\end{gather}
These optimality conditions furnish fundamental insight and information into the solutions of (\ref{probgenn}). Under standard sets of constraint qualifications, solutions of optimality conditions furnish (local) solutions of (\ref{probgenn}). 
They are the guiding principle to design numerical algorithms to approximate those solutions. They are also the starting point of duality theory so fundamental to the understanding of mathematical programming. As it is well-known, the basic, rough idea of duality is to set up a new mathematical program, intimately connected to (\ref{probgenn}), and, in particular, designed with the main ingredients of that (primal) problem, but in which multipliers $(\by, \bz)$ play a central role in the form of dual variables. There is an intimate relationship between optimal solutions $\overline\bx$ of the primal, and optimal solutions $(\overline\by, \overline\bz)$ of the dual. 

Our point of view here is a bit different, and though it also deals with multipliers $(\by, \bz)$, we seek them, in association with the primal variable $\bx$,  not in the form of the optimal solution of another (dual) mathematical program, but rather as a fixed point of a suitable map. What is more important, the structure of that map is such that, under mild assumptions, the typical procedure consisting in iterating the action of such a map, converges to the optimal triplet $(\overline\bx, \overline\by, \overline\bz)$. It is remarkable that the map is so simple to define, and so easy to implement in practice, if we can rely on a efficient procedure for unconstrained optimization. 
The strategy of  using free, unconstrained programs to approximate the optimal solutions of general constrained optimization problems like (\ref{probgenn}) is quite natural, and appealing. This is the source of many fundamental algorithms utilized today. It will also be a main inspiration for us here. 

As a matter of fact, the method we want to examine is designed to deal just with inequality constraints so that there is no map $\bbh$ in (\ref{probgenn})
\begin{equation}\label{probgener}
\hbox{Minimize in }\bx\in\R^N:\quad f(\bx)\quad\hbox{subject to}\quad \bbg(\bx)\le\bcero,
\end{equation}
for a smooth, real function $f:\R^N\to\R$, and a smooth, vector-valued mapping $\bbg:\R^N\to\R^m$. This is not, in principle, a significant limitation because the equality constraint $\bbh(\bx)=\bcero$ can be, equivalently, translated  into $\bbh(\bx)\le\bcero$ together with $-\bbh(\bx)\le\bcero$.

A central role in our method is played by the master function
\begin{equation}\label{maestra}
L(\bx, \by)=f(\bx)+\sum_{k=1}^m e^{y^{(k)}g_k(\bx)},\quad
\by=\left(y^{(k)}\right)_{k=1, 2, \dots, m},\quad \bbg=\left(g_k\right)_{k=1, 2, \dots, m}.
\end{equation}
It is definitely reminiscent of the typical lagrangean for dual theory, though at the same time it is a bit different. We will understand soon the main reasons that support such choice. Notice that this master function is quite different from $f(\bx)+\exp{(\by\cdot\bbg(\bx))}$. We will comment on this later. 

Suppose that $L(\bx, \by)$ is strictly convex, and coercive in $\bx$ for every choice $\by\in\R^m_+$ of vectors with (strictly) positive coordinates. Our basic map 
$$
\bG(\by):\overline{\R^m_+}\mapsto\overline{\R^m_+}
$$
is the result of the composition of two operations:
\begin{enumerate}
\item For given $\by\in\R^m_+$, find (approximate) the (global) solution of the unconstrained problem
$$
\hbox{Minimize in }\bx\in\R^N:\quad L(\bx, \by).
$$
The passage from $\by$ to $\bx\equiv\bx(\by)$ is, therefore, a well-defined and smooth operation, under smoothness conditions for all the ingredients of the original problem.
\item Put
$$
\bG=\left(G_k\right)_{k=1, 2, \dots, m},\quad G_k(\by)=y^{(k)}e^{y^{(k)}g_k(\bx(\by))},
$$
when $\by\in\R^m_+$, and extend it by continuity for $\by\in\overline{\R^m_+}$. Here we take
$$
\R^m_+=\{\by\in\R^m: y^{(k)}>0\}
$$
while
$$
\overline\R^m_+=\{\by\in\R^m: y^{(k)}\ge0\}.
$$
\end{enumerate}
This extension by continuity deserves some comments. On the one hand, note that if some component $k$ of $\by$ vanishes, $y^{(k)}=0$, then trivially $G_k(\by)=0$ regardless of the value of $g_k(\by)$. This is not convenient, since $y^{(k)}=0$ must be somehow related to the constraint $g_k(\by)\le0$. On the other, notice that, after all, the constraint $g_k(\bx)\le0$ is equivalent to $y^{(k)}g_k(\bx)\le0$ for every positive $y^{(k)}$, but there is definitely a discontinuity if we set $y^{(k)}=0$, for then the constraint drops out. In other words, the optimization problem
$$
\hbox{Minimize in }\bx\in\R^N:\quad f(\bx)\quad\hbox{subject to}\quad y^{(k)} g_k(\bx)\le0\hbox{ for all }k, 
$$
for a fixed vector $\by$ with strictly positive components $y^{(k)}$ is equivalent to (\ref{probgener}). However, if some of the components of $\by$ vanish, then the corresponding constraint drops out, and so there is clearly a lack of continuity. 

We assume, for the time being, that this extension is possible. If it is so, then the values of $\bG(\by)$ for $\by\in\overline{\R^m_+}\setminus\R^m_+$ are obtained by taking limits of $\bG(\by_j)$ when $\by_j\to\by$, and $\by_j\in\R^m_+$. It is important to stress this fact because it amounts to a certain stability of the map $\bG$ at those vectors $\by\in\overline{\R^m_+}$. 

A somewhat surprising fact that places this map $\bG$ into perspective is the following.
\begin{proposition}\label{primera}
Under the assumption that $L(\bx, \by)$ is strictly convex, and coercive in $\bx$ for all $\by\in\R^m_+$, suppose a certain vector $\overline\by\in\overline{\R^m_+}$ is a locally stable, fixed point for $\bG$ in the sense that the reiteration of the action of $\bG$ starting in a vicinity of $\overline\by$ converges to $\overline\by$. Then $\overline\bx=\bx(\overline\by)$ is a (local) solution of (\ref{probgener}), even if it is not a solution of the corresponding optimality conditions \eqref{optim}.
\end{proposition}
\begin{proof}
If $\overline\by$ is a true fixed point for $\bG$, then we should have
$$
\overline y^{(k)}=\overline y^{(k)}e^{\overline y^{(k)}g_k(\overline\bx)}.
$$
This equation amounts to two possibilities: either $\overline y^{(k)}=0$, or else $1=e^{\overline y^{(k)}g_k(\overline\bx)}$, i.e., $\overline y^{(k)}g_k(\overline\bx)=0$.  At any rate, $\overline y^{(k)}g_k(\overline\bx)=0$ for all $k$. Assume $\overline y^{(k)}$ vanishes. Then, as emphasized earlier just before the statement of this proposition, there is some stability of $\bG$ around $\overline\by$ in the sense that $\bG$ is defined over $\overline{\R^m_+}\setminus\R^m_+$ through continuous extension,  
in such a way that if $g_k(\overline\bx)>0$, by continuity
$$
g_k(\bx(\by))>0, \quad y^{(k)}>0,
$$
for $\by$ in small neighborhoods of $\overline\by$. But then
$$
G_k(\by)=y^{(k)}e^{y^{(k)}g_k(\bx(\by))}>y^{(k)}
$$
systematically in such a vicinity, and such $\overline\by$ could not be a fixed point of $\bG$ found by reiteration of the action of $\bG$. Therefore $\overline\by$ with $\overline y^{(k)}=0$ cannot be a fixed point for $\bG$ unless $g_k(\overline\bx)\le0$. This argument implies that $\overline\bx$ is indeed feasible for problem \eqref{probgener}. 

Suppose now that we could find a vector $\tilde\bx$, not far from $\overline\bx$, such that $f(\tilde\bx)<f(\overline\bx)$ and $\bbg(\tilde\bx)\le\bcero$. Because $\overline\by\ge\bcero$, we would have $\by^{(k)} g_k(\tilde\bx)\le0=\overline y^{(k)} g_k(\overline\bx)$ for all $k$. Hence, it is clear that
$$
L(\tilde\bx, \overline\by)<L(\overline\bx, \overline\by),
$$
but this contradicts the very nature of $\overline\bx$ as a local minimum for $L(\cdot, \overline\by)$. This contradiction proves the statement. 
\end{proof}

The numerical procedure that arises from this perspective is amazingly simple to implement, but relies in a fundamental way on being capable of efficiently approximating the map $\bx(\by)$. This amounts to unconstrained optimization. It reads:
\begin{enumerate}
\item Initialization. Take  $\by_0\in\R^m_+$, $\by_0>\bcero$, and $\bx_0\in\R^N$ in an arbitrary way, or appropriately located in a certain valley. For instance, $\by_0=\buno$, $\bx_0=\bcero$. 
\item Iterative step until convergence. Suppose we have computed $\by_j$, and $\bx_j$. 
\begin{enumerate}
\item Solve for the unconstrained optimization problem 
$$
\hbox{Minimize in }\bz\in\R^N: \quad L(\bz, \by_j)=f(\bz)+\sum_{k=1}^me^{\by_j^{(k)} g_k(\bz)}
$$
starting from the intial guess $\bx_j$. Let $\bx_{j+1}$ be such (local) minimizer.
\item If $\by_j\cdot \bbg(\bx_{j+1})$
vanishes (i.e. is reasonably small), stop, take $\by_j$ as the multiplier of the problem, and $\bx_{j+1}$ as the solution of the constrained problem.
\item If $\by_j\cdot \bbg(\bx_{j+1})$ does not vanish, update 
\begin{equation}\label{actualizar}
y_{j+1}^{(k)}=e^{y_j^{(k)} g_k(\bx_{j+1})}y_j^{(k)}
\end{equation}
for all $k=1, 2, \dots, m$. 
\end{enumerate}
\end{enumerate}

The intuition after this algorithm is pretty clear. In Section \ref{dos}, we try to justify why it is plausible to expect that this algorithm should furnish, at least, reasonable results. Each value $\by\in\R^m_+$ establishes an exponential barrier for the constraints, in such a way that if the minimizer $\bx(\by)$ turns out to be non-feasible, then the barrier should be intensified. This is what the update rule (\ref{actualizar}) does in that case. If, on the other hand, the constraint is met, then the barrier should be relaxed so as letting the objective function $f$ to seek ``more freely without restriction" its minimum. This is again accomplished by the update rule. 

This is a good place to stress how the form of $L(\bx, \by)$ cannot be $f(\bx)+\exp(\by\cdot\bbg(\bx))$, because for this other choice, the update rule for the auxiliary variable $\by$ 
$$
\by_{j+1}=e^{\by_j\cdot\bbg(\bx_{j+1})}\by_j
$$
would be the same for all components, and this is too rigid to work well: each component $k$ should adapt to its corresponding constraint separately from the others. 

Beyond the convergence theorems that we will prove, Proposition \ref{primera} is a  clear and powerful statement. In practice, without any further concern about assumptions, one can use the above algorithm. If the variables $\bx$ and $\by$ do converge, the limit vector $\overline\bx$ has to be a (local) solution of (\ref{probgener}).  Indeed, the algorithm is quite flexible to the point that the set-valued map $\by\mapsto \bX(\by)$, where $\bX(\by)$ stands for the full set of local minima of $L(\cdot, \by)$, admits selections to approximate all of the (isolated) local minima of (\ref{probgenn}). 

Our main task here focuses on showing that, under appropriate standard hypotheses, this algorithm always converges to minima of the underlying constrained problem (\ref{probgener}). As a matter of fact, we introduce a main assumption that pretends to avoid singular situations. It plays the role of a certain constraint qualification, as it expresses the idea that local minima of the master function are feasible once the exponential barriers are sufficiently large. 

\begin{definition}\label{wb}
We say that problem \eqref{probgener} is well-balanced if for all $\by\in\R^m_+$ with all the components $y^{(k)}$ sufficiently large, we have $\bbg(\bx(\by))\le\bcero$. 
\end{definition}

Our main results follows. 

\begin{theorem}\label{principalg}
Suppose the cost function $f:\R^N\to\R$, and the  components of the constraint map $\bbg:\R^N\to\R^m$ determining problem \eqref{probgener} comply with:
\begin{enumerate}
\item they all are smooth, and convex;
\item the corresponding $L(\bx, \by)$ is coercive in $\bx$ for every fixed $\by$ with positive components;
\item the problem is well-balanced according to Definition \ref{wb}.
\end{enumerate}
Then the above algorithm always converges to a (global) minimizer for (\ref{probgener}).
\end{theorem}

The property of being well-balanced, though suitable for the proof of Theorem \ref{principalg},  may be hard to check in practice. In some cases, it is impossible because it is not correct. We will actually see that our way of dealing with equality constraints leads to a situation where the resulting problem cannot be well-balanced. Fortunately, there is a version of Proposition \ref{primera} that allows for the possibility of having some components of vectors $\by$ go to infinity in the above iterative process, as long as we keep under control the products $y^{(k)}g_k(\bx(\by))$ as the iterations proceed. 

By removing the convexity conditions, we are typically left with a local convergence theorem. Different local solutions are reached by different initializations in the algorithm, as remarked above. The practical implication of our analysis, as suggested earlier, is that optimal solutions of mathematical programs will be captured by this algorithm whenever they exist. 

The main goal of the paper is, in addition to introducing the algorithm itself, to prove the convergence  result Theorem \ref{principalg}.  We will proceed through several steps of increasing generality. 

The kind of exponential penalty functions used in a fundamental way in this contribution have been utilized and described before, but in a standard context taken as barriers. See for instance \cite{bertsekas}. Apparently (\cite{fiaccomccormick}), they first were considered by T. S. Motzkin who, in 1952,  suggested the use of exponentials for satisfying a system of linear inequalities. 
One of our favorites sources for numerical optimization is \cite{nocedal}.

\section{A perspective for constrained problems based on unconstrained minimization}\label{dos}
We will start with the simple basic problem
\begin{equation}\label{probbas}
\hbox{Minimize in }\bx\in\R^N:\quad f(\bx)\quad\hbox{ subject to }\quad g(\bx)\le0,
\end{equation}
where  both $f$, and $g$ are smooth functions. We therefore have a single inequality constraint. 

We would like to design an iterative procedure to approximate solutions for (\ref{probbas}) in an efficient, practical, accurate way. We introduce our initial demands in the form:
\begin{enumerate}
\item The main iterative step is to be an unconstrained minimization problem, and as such its corresponding objective function must be defined in all of space (exterior point methods). 
\item More specifically, we would like to design a real function $H$ so that the main iterative step of our procedure be applied to the augmented cost function $L(y, \bx)=f(\bx)+H(y g(\bx))$ for the $\bx$-variable. We hope to take advantage of the joint dependence upon $y$ and $\bx$ inside the argument for $H$. Notice that letting $y$ out of $H$ may not mean a real change as we would be back to (\ref{probbas}) with a $g$ which would be the composition $H(g(\bx))$. 
\item The passage from one iterative step to the next is performed through an update step for the variable (multiplier) $y$.
\item Convergence of the scheme should lead to a solution of (\ref{probbas}). 
\end{enumerate}
Notice that 
\begin{equation}\label{gradiente}
\nabla_\bx L(y, \bx)=\nabla f(\bx)+H'(y g(\bx))y\nabla g(\bx),
\end{equation}
and that optimality conditions for (\ref{probbas}) read 
\begin{equation}\label{optimalidad}
\nabla f(\bx)+y\nabla g(\bx)=0, y g(\bx)=0, \quad y\ge0, g(\bx)\le0.
\end{equation}
Thus each main iterative step enforces the main equation in (\ref{optimalidad}), the one involving gradients and derivatives. But we would like to design the function $H(t)$ to ensure that as a result of the iterative procedure, the other conditions in (\ref{optimalidad}) are also met.

The following features seem to be very convenient:
\begin{enumerate}
\item Variable $\bx$ will always be a solution of 
$$
\nabla f(\bx)+H'(y g(\bx))y\nabla g(\bx)=\bcero.
$$
\item Variable $y$ will always be non-negative (in practice strictly positive but possibly very small). Comparison of (\ref{gradiente}) with (\ref{optimalidad}) leads to the identification
$y\mapsto H'(y g(\bx))y$, and so we would like $H'\ge0$. 
\item The multiplier $H'(y g(\bx))y$ can only vanish if $y$ does. The update rule for the variable $y$  should be $y\mapsto H'(y g(\bx))y$. If at some step we hit  the true value of the multiplier $y$, then simultaneously $y g(\bx)=0$, and so
we would also like to have $H'(0)=1$.
\item If $g(\bx)>0$, then the update rule above for $y$ must yield a higher value for $y$ so as to force in a more intense way in the next iterative step the feasible inequality $g\le0$. Hence, $H''>0$, or $H$, convex (for positive values). In addition, $H''(t)\to+\infty$ when $t\to+\infty$. 
\item If $g(\bx)$ turns out to be (strictly) negative, then we would like $y$ to become smaller so as to let the minimization of $f$ proceed with a lighter interference from the inequality constraint. This again leads to $G$ convex (for negative values).
\item The optimality condition $y g(\bx)=0$ becomes $H'(y g(\bx))y g(\bx)=0$. Thus the function $H'(t)t$ can only vanish if $t=0$. In particular, $H'>0$. 
\end{enumerate}

All of these reasonable conditions impose the requirements
$$
H>0, H'>0, H'(0)=1, H''>0, H''(t)\to\infty, \hbox{ if }t\to\infty.
$$
Possibly, the most familiar choice if $H(t)=e^t$, and this is the one we will select. 

The iterative procedure is then as follows.
\begin{enumerate}
\item Initialization. Take  $y_0>0$, and $\bx_0\in\R^N$ in an arbitrary way. For instance, $y_0=1$, $\bx_0=\bcero$. 
\item Iterative step until convergence. Suppose we have $y_j$, $\bx_j$. 
\begin{enumerate}
\item Solve for the unconstrained optimization problem 
$$
\hbox{Minimize in }\bz\in\R^N: \quad f(\bz)+e^{y_j g(\bz)}
$$
starting from the intial guess $\bx_j$. Let $\bx_{j+1}$ be such (local) minimizer.
\item If $y_jg(\bx_{j+1})$
vanishes, stop: take $y_j$ as the multiplier of the problem, and $\bx_{j+1}$ as the solution of the constrained problem.
\item If $y_jg(\bx_{j+1})$ does not vanish, update 
$$
y_{j+1}=e^{y_j g(\bx_{j+1})}y_j.
$$
\end{enumerate}
\end{enumerate}

\section{Some preliminaries}
According to Proposition \ref{primera}, if we are interested in a convergence theorem to a solution of \eqref{probbas}, 
all we need to care about is to ensure the smoothness and convexity conditions for the master function $L(\bx, \by)$, and then show convergence of our algorithm to a fixed point of the map $\bG$. One can envision to write down some hypotheses so that there is some $M>0$ in such a way that 
$$
\bG:[0, M]^m\mapsto[0, M]^m.
$$
In this case, Brower's fixed point theorem would let us conclude indeed the existence of fixed points for $\bG$. However, that would not imply, in principle, a convergence result. At any rate, one needs to argue first about the continuous extension of $\bG$ to all of $\overline{\R^m_+}$. 
\begin{proposition}
Suppose all functions involved in \eqref{probgener} are smooth, and convex, and that $L(\bx, \by)$ is coercive, and strictly convex in $\bx$, uniformly in $\by$, for all $\by\in\R^m_+$. Define the map $\bG(\by):\R^m_+\to\R^m_+$ as indicated above. Then $\bG$ is locally Lipschitz continuous, and can be extended in a unique, continuous way to $\bG:\overline{\R^m_+}\mapsto\overline{\R^m_+}$ 
\end{proposition}
\begin{proof}
Under the assumptions written in the statement, the map taking each vector $\by\in\R^m_+$ into the global minimizer $\bx(\by)$ of $L(\bx, \by)$ is well-defined and smooth. As a matter of fact, we have the system 
\begin{equation}\label{sistemaa}
\nabla f(\bx(\by))+\sum_{k=1}^m y^{(k)}e^{y^{(k)}g_k(\bx(\by))} \nabla g_k(\bx(\by))=\bcero,
\end{equation}
determining implicitly, in a unique way, the vector $\bx=\bx(\by)$. The Implicit Function Theorem implies then that the dependence $\by\mapsto\bx(\by)$ is smooth, and, in particular, locally Lipschitz continuous.   Notice how the gradient of $\nabla_\bx L(\bx, \by)$, which is the right-hand side of \eqref{sistemaa}, with respect to $\bx$ is the hessian of the master function $L(\bx, \by)$ with respect to $\bx$, and so it is non-singular due to the uniform strict convexity assumed in the statement. 
As a consequence, the composition defining the mapping $\bG(\by)$ is also (locally) Lipschitz continuous. Finally, notice that this lipschitzianity implies the continuous extension, in a unique way, of $\bG$ to $\overline{\R^m_+}$. 
\end{proof}

Under the main hypothesis contained in Definition \ref{wb}, it turns out that the mapping $\bG$ can be regarded as a map from a certain cube into itself. 
\begin{proposition}
Suppose, in addition to the hypotheses of the previous proposition, that problem \eqref{probgener} is well-balanced. Then there is $M>0$ such that
$$
\bG:[0, M]^m\mapsto[0, M]^m.
$$
\end{proposition}
We will not prove here this proposition because of two main reasons. On the one hand, we argue in the next paragraph that the existence of fixed points for $\bG$ does not imply a convergence theorem, which is the main fact we are after. On the other, the proof of this result will become pretty clear when we prove our main convergence theorem below. 

As a consequence of these two last propositions, the map $\bG$, under the appropriate assumptions, admits fixed points. By Proposition \ref{primera}, if $\overline\by$ is such a fixed point, then the vector $\overline\bx=\bx(\overline\by)$ is a solution of \eqref{probgener}. If we, a priori, know that there cannot be more than one solution for \eqref{probgener}, because of some strict convexity for instance, then the map $\bG$ cannot have more than one such fixed point. But even so, we cannot be sure if the iterates produced by our scheme (or whatever procedure) will converge to such a fixed point. There is no substitute for a convergence theorem in which we very closely analyze the behavior of the algorithm. We will therefore focus, without mentioning it explicitly anymore, on showing that iterates converge to fixed points $\overline\by$ of $\bG$, and then Proposition \ref{primera} permits us conclude that the corresponding vectors $\overline\bx$ are the sought solution of \eqref{probgener}.

Although we are asking for main structural hypotheses, in the form of convexity, for the functions determining problem \eqref{probgener} to ensure that the map taking each vector $\by\in\R^m_+$ into the minimum of the master function $L(\bx, \by)$ is single-valued, well-defined, and smooth, when those do not hold, we would have a set-valued mapping $\bX(\by)$ of all local minima of the master function. 
We would be left with a convergence fact for local minima as usual. Indeed, the initialization point for each iteration in our basic procedure determines a suitable continuous selection that leads, when convergence takes place, to a local optimal solution of the problem. 

We finally elaborate a bit on Definition \ref{wb}. It does not look sufficiently explicit for practical purposes.  A better criterium in this regard is the following. It is written in the spirit of a constraint qualification.
\begin{lemma}\label{well}
Suppose the cost function $f$, and the functions $g_i$ determining the constraints are such that, in addition to being smooth:
\begin{enumerate}
\item there is $r>0$ with the property that for every $\bx$, if $I(\bx)$ is the set of indices $j$ where $g_j(\bx)\ge-r$, then the gradients $\{\nabla g_j(\bx): \bx\in I(\bx)\}$ are positively independent;
\item there is a positive constant $M$ such that
$$
\frac1M\le\frac{|\nabla g_i(\bx)|}{|\nabla g_j(\bx)|}\le M
$$
for every $i$ and $j$, and $\bx$, and 
$$
|\nabla f(\bx)|\le M|\nabla g_k(\bx)|
$$
for every $k$, and $\bx$. 
\end{enumerate}
Then the corresponding mathematical problem is well-balanced according to Definition \ref{wb}.
\end{lemma}
\begin{proof}
First-order optimality conditions for the (unconstrained) master problem read
$$
\nabla f(\bx)+\sum_{k=1}^my^{(k)}e^{y^{(k)}g_k(\bx)}\nabla g_k(\bx)=\bcero.
$$
They can be recast in the form
$$
\nabla f(\bx)+\sum_{k\notin I(\bx)}y^{(k)}e^{y^{(k)}g_k(\bx)}\nabla g_k(\bx)=-\sum_{k\in I(\bx)}y^{(k)}e^{y^{(k)}g_k(\bx)}\nabla g_k(\bx).
$$
Put $J(\bx)$ for the set of indices $k$ such that $g_k(\bx)\ge0$. Note that $J(\bx)\subset I(\bx)$. 
If for some $\by'$s arbitrarily large, $J(\bx)$ is non-empty, 
the right-hand side of the previous vector equality would tend to infinity because the positive independence assumed on the gradients prevents from having cancellations, and at least one component (if $J(\bx)$ is non-empty) would be arbitrarily large; but on the other, the left-hand side is definitely bounded regardless of the size of $\by$ because the coefficient in front of $\nabla g_k$ would be smaller than 
$y^{(k)}e^{-y^{(k)}r}$ with $r>0$.  This contradiction implies that the set of indices $J(\bx)$ has to be empty for $\by$ sufficiently large, and this means that the problem is well-balanced. 
\end{proof}

\section{A convergence theorem}\label{diuno}
We would like to provide a solid foundation for our approximation procedure for mathematical programs by proving convergence theorems as the following. We start with the one-constraint situation. Recall that the master function is
$$
L(\bx, y)=f(\bx)+e^{yg(\bx)}.
$$
Though the treatment of this one-dimensional situation is elementary, it will be the basic building block upon which  the general, multidimensional case will be shown. The hypothesis of the problem being well-balanced is changed by a much weaker one. 

\begin{theorem}\label{onedimension}
Suppose the cost function $f:\R^N\to\R$, and the constraint function $g:\R^N\to\R$ satisfy the following requirements:
\begin{enumerate}
\item they are smooth, and convex;
\item the corresponding $L(\bx, y)$ is coercive in $\bx$ for every fixed positive $y$;
\item there is some $\overline y\ge0$ such that $g(\bx(\overline y))\le0$.
\end{enumerate}
Then the above algorithm always converges to a (global) minimizer for (\ref{probbas}).
\end{theorem}
\begin{proof}
By a standard perturbation argument depending on a small parameter $\epsilon>0$, we can assume, without loss of generality, that the convexity condition imposed on $f$ is strict, and that the hessian $\nabla^2f(\bx)$ is a symmetric, positive definite matrix for all $\bx$. It suffices to add a term like $(\epsilon/2)|\bx|^2$ to $f(\bx)$. 
Under this strengthened hypothesis, the master function $L$, regarded as a function of $\bx$, is a coercive, strictly convex function, and so the unique minimizer $\bx\equiv\bx(y)$ is determined implicitly through the (unique) solution of the non-linear system
\begin{equation}\label{optimal}
\nabla f(\bx)+e^{yg(\bx)}y\nabla g(\bx)=\bcero,
\end{equation}
and so it is smooth by the Implicit Function Theorem. As already pointed out, the gradient of (\ref{optimal}) with respect to $\bx$ (the hessian of $L(\cdot, y)$) cannot be singular precisely because $L(\bx, y)$ is strictly convex with respect to $\bx$.
The conclusion of the statement of Theorem \ref{onedimension} will be a direct consequence of two lemmae whose proofs rely on suitable manipulations of (\ref{optimal}). All functions involved are smooth. 
\begin{lemma}\label{aux}
For every positive $y$, 
$$
[1+y g(\bx(y))]\frac d{dy}g(\bx(y))\le0.
$$
\end{lemma}
\begin{proof}
Since $\bx(y)$ is determined as the result of a (local) minimization process, we also have, in addition to (\ref{optimal}), 
$$
\nabla^2f(\bx)+e^{yg(\bx)}y^2\nabla g(\bx)\otimes\nabla g(\bx)+e^{yg(\bx)}y\nabla^2g(\bx)\ge0
$$
as symmetric matrices. This condition is nothing but the positivity of the hessian (with respect to $\bx$) as has already been indicated above. 
In particular, since all functions are smooth, $\bx'(y)$ is well-defined, and
\begin{equation}\label{segunda}
\bx'\cdot\nabla^2f(\bx)\bx'+e^{yg(\bx)}y^2|\nabla g(\bx)\bx'|^2+e^{yg(\bx)}y\bx'\cdot\nabla^2g(\bx)\bx'\ge0.
\end{equation}
On the other hand, from
$$
\nabla f(\bx(y))+e^{yg(\bx(y))}y\nabla g(\bx(y))=\bcero,
$$
differentiating with respect to $y$, we arrive at
$$
\nabla^2f(\bx)\bx'+e^{yg(\bx)}[g(\bx)+y\nabla g(\bx)\bx']y\nabla g(\bx)+
e^{yg(\bx)}[\nabla g(\bx)+y\nabla^2g(\bx)\bx']=\bcero.
$$
Multiplying by $\bx'$, and comparing to (\ref{segunda}), we see that
$$
e^{yg(\bx)}[yg(\bx)+1]\nabla g(\bx)\bx'\le0.
$$
This is exactly the statement in the lemma.
\end{proof}

Let us now focus on the function $G(y):\R^+\to\R^+$ given by 
$$
G(y)=e^{yg(\bx(y))}y
$$
where $\bx(y)$ is again determined by (\ref{optimal}), and $y>0$. 

\begin{lemma}\label{auxx}
The function $G$ is smooth, and for every positive $y$, 
$$
G'(y)\frac d{dy}g(\bx(y))\le0.
$$
\end{lemma}
\begin{proof}
 It is clear that (\ref{optimal}) can be written as 
$$
\nabla f(\bx(y))+G(y)\nabla g(\bx(y))=\bcero.
$$
By differentiating with respect to $y$, we will have
\begin{equation}\label{derivadaa}
\nabla^2f(\bx)\bx'(y)+G'(y)\nabla g(\bx)+G(y)\nabla^2 g(\bx)\bx'(y)=\bcero,
\end{equation}
and
$$
\bx'(y)\cdot\nabla^2f(\bx)\bx'(y)+G'(y)\bx'(y)\cdot\nabla g(\bx)+G(y)\bx'(y)\cdot\nabla^2 g(\bx)\bx'(y)=\bcero.
$$
Therefore
$$
G'(y)\frac d{dy}g(\bx(y))=-\bx'(y)\cdot[\nabla^2f(\bx)+G(y)\nabla^2 g(\bx)]\bx'(y)\le0
$$
due to the convexity assumed on $f$, and $g$.
\end{proof}
If we further put, for simplicity, $\overline g(y)=g(\bx(y))$, we have the three properties
\begin{equation}\label{desigualdades}
(1+y\overline g(y))\overline g'(y)\le0,\quad G'(y)\overline g'(y)\le0, \quad G'(y)(1+y\overline g(y))\ge0.
\end{equation}
Notice that the third one is a consequence of the other two, which are the conclusion of the two lemmae above. 
From these properties, we would like to highlight the following consequences:
\begin{enumerate}
\item If $\overline g(\overline y)\le0$ for some $\overline y\ge0$, then $\overline g(y)\le0$ for all $y\ge\overline y$. This is a direct consequence of the first inequality in (\ref{desigualdades}). Indeed, suppose there is  some $y>\overline y$ with $\overline g(y)>0$, and put
$$
y_0=\inf\{y>\overline y: \overline g(y)>0\}.
$$
It is then clear that $\overline g(y_0)=0$, and $\overline g'(y_0)\ge0$.
But then, again that first inequality in (\ref{desigualdades}) would imply $\overline g(y_0)\le-1/y_0$, a contradiction with the fact $\overline g(y_0)=0$. The argument is standard. 
As a consequence if $G(\overline y)\le\overline y$ for some $\overline y\ge0$, then $G(y)\le y$ for all $y\ge\overline y$. 
\item Over the set $\overline g\ge0$, i.e. over the set where $G$ is greater than or equal to the identity, $G$ is non-decreasing (third inequality in (\ref{desigualdades})), and $\overline g$ is non-increasing (first inequality in (\ref{desigualdades})).
\end{enumerate}

\begin{center}
\includegraphics[width=6.5cm]{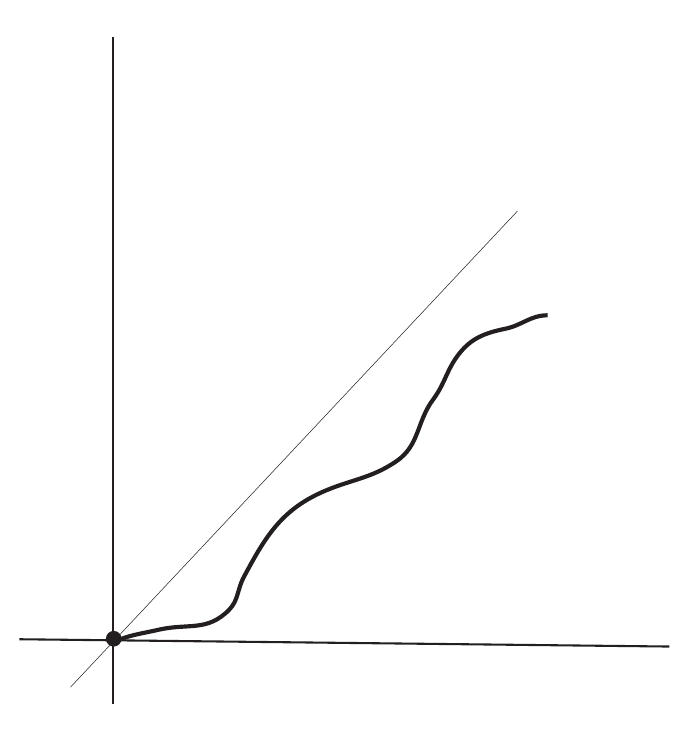}\quad
\includegraphics[width=6.5cm]{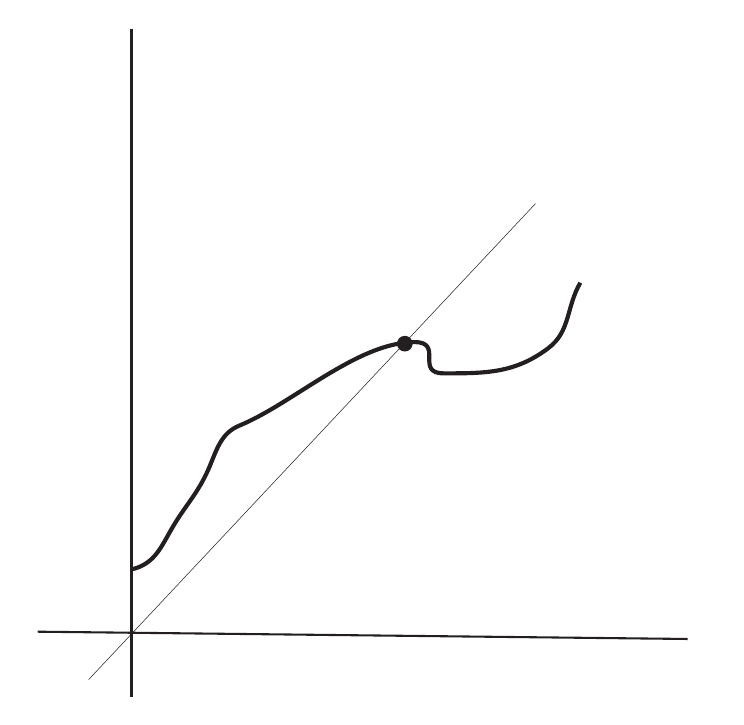}

Figure 1. The one-dimensional situation.  
\vspace*{0.5cm} 
\end{center}

We can now have two main scenarios. Because of our hypothesis on the existence of a certain positive $\overline y$ with $\overline g(\overline y)\le0$, we conclude as above that $G(y)\le y$ for all $y\ge\overline y$. 
\begin{enumerate}
\item It could happen that $G(y)\le y$, for all $y\ge0$ (Figure 1, left picture). In this case, $G$ has at least one fixed point at $y=0$, and the algorithm converges to it in a neighborhood of it. This corresponds to the situation where the global minimizer of $f$ complies with $g\le0$, and the restriction is inactive. 
\item Assume there is some $y_0>0$, so that $G(y_0)>y_0$ (Figure 1, right picture). By the properties just written, $G$ will have at least one fixed point in the interval $(y_0, \overline y]$, and the iteration process within this interval will always, regardless of the starting point in this same interval, converge to one such fixed point of $G$. In this case $\overline g$ will vanish at such fixed point, and the restriction is active. 
\end{enumerate}
\end{proof}

Even though the third hypothesis assumed in Theorem \ref{onedimension} is weaker than the well-balanced condition, in practice one would have to check this last requirement which in the one-restriction case simplifies to the two properties
\begin{gather}
|\nabla f|\le M|\nabla g|\hbox{ for some }M,\nonumber\\
\nabla g(\bx)=\bcero\hbox{ can only occur when }g(\bx)\le -r\nonumber
\end{gather}
for some fixed $r>0$. 

\section{The case of multiple constraints}\label{secmul}
Once the nature of the algorithm we would like to propose has been clarified with a single inequality constraint, we want to  examine the same strategy for several such conditions. It will suffice to focus on just two such constraints to be able to figure out the situation with many more such inequality constraints. 

Consider the optimization problem
\begin{equation}\label{mathprog}
\hbox{Minimize in }\bx\in \R^N:\quad f(\bx)\quad\hbox{ subject to }\quad g_1(\bx)\le0, g_2(\bx)\le0.
\end{equation}
After the situation examined above, we focus on the unconstrained problem
$$
\hbox{Minimize in }\bx\in\R^N:\quad f(\bx)+e^{y_1g_1(\bx)}+e^{y_2g_2(\bx)},
$$
and regard the (local) solution $\bx\equiv\bx(y_1, y_2)$ for $y_i>0$ as an implicit mapping depending on $\by=(y_1, y_2)$. We assume that $f$ as well as $g_i$ are such that there is no difficulty in finding $\bx(y_1, y_2)$, and that this dependence is continuous, even smooth. After Lemma \ref{aux}, we would like to use the fundamental information
$$
\nabla f(\bx)+e^{y_1g_1(\bx)}y_1\nabla g_1(\bx)+e^{y_2g_2(\bx)}y_2\nabla g_2(\bx)=\bcero,
$$
together with the local convexity condition
\begin{align}
\nabla^2 f(\bx)&+e^{y_1g_1(\bx)}[y_1^2\nabla g_1(\bx)\otimes\nabla g_1(\bx)+y_1\nabla^2g_1(\bx)]\nonumber\\
&+e^{y_2g_2(\bx)}[y_2^2\nabla g_2(\bx)\otimes\nabla g_2(\bx)+y_2\nabla^2g_2(\bx)]\ge\bcero,\nonumber
\end{align}
in the sense of symmetric matrices. One would have then to fix a unit direction $\bn=(n_1, n_2)$, and try to find relevant information much in the same way as we have done with the single inequality situation. We believe, however, that it is a much more transparent strategy to freeze alternatively each one of the two multiplier $y_i$, $i=1, 2$, and apply the single inequality constraint with respect to the complementary constraint. 

Namely, let $y_1\ge0$ be fixed, and consider the mathematical program
$$
\hbox{Minimize in }\bx\in\R^N:\quad f(\bx)+e^{y_1g_1(\bx)}\quad\hbox{ subject to }\quad g_2(\bx)\le0.
$$
Assume that the hypotheses of Theorem \ref{onedimension} permit us to conclude that there is $\bx^{(1)}\equiv\bx(y_1)$, and $y_2^{(1)}\equiv y_2(y_1)$ such that
\begin{gather}
\nabla f(\bx^{(1)})+e^{y_1g_1(\bx^{(1)})}y_1\nabla g_1(\bx^{(1)})+y_2^{(1)}\nabla g_2(\bx^{(1)})=\bcero,\nonumber\\
y_2^{(1)}\ge0,\quad g_2(\bx^{(1)})\le0,\quad y_2^{(1)}g_2(\bx^{(1)})=0.\nonumber
\end{gather}
Likewise, we would also have $\bx^{(2)}\equiv\bx(y_2)$, and $y_1^{(2)}\equiv y_1(y_2)$ such that
\begin{gather}
\nabla f(\bx^{(2)})+y_1^{(2)}\nabla g_1(\bx^{(2)})+e^{y_2g_2(\bx^{(2)})}y_2\nabla g_2(\bx^{(2)})=\bcero,\nonumber\\
y_1^{(2)}\ge0,\quad g_1(\bx^{(2)})\le0,\quad y_1^{(2)}g_1(\bx^{(2)})=0.\nonumber
\end{gather}
This solution would correspond to the problem
$$
\hbox{Minimize in }\bx\in\R^N:\quad f(\bx)+e^{y_2g_2(\bx)}\quad\hbox{ subject to }\quad g_1(\bx)\le0,
$$
again through Theorem \ref{onedimension}, assuming that the appropriate hypotheses hold. The whole argument then revolves around ensuring that the graphs of the two functions $y_2(y_1)$, and $y_1(y_2)$ meet at some pair $\overline\by=(\overline y_1, \overline y_2)$, for in this case we would have a vector 
$$
\overline\bx=\bx(\overline y_1, \overline y_2)=\bx^{(1)}(\overline y_1)=\bx^{(2)}(\overline y_2)
$$ 
such that
\begin{gather}
\nabla f(\overline\bx)+\overline y_1\nabla g_1(\overline\bx)+\overline y_2\nabla g_2(\overline\bx)=\bcero,\label{optimalidaddos}\\
\overline\by\ge\bcero,\quad \bbg(\overline\bx)\le\bcero,\quad\overline\by\cdot\bbg(\overline\bx)=0,\quad 
\overline\by=(\overline y_1, \overline y_2),\nonumber
\end{gather}
the optimality conditions for a solution of the mathematical program (\ref{mathprog}). 
The iterative procedure can be set up in a similar way.
\begin{enumerate}
\item Initialization. Take  $\by_0>\bcero$, and $\bx_0\in\R^N$ in an arbitrary way. For instance, $\by_0=\buno$, $\bx_0=\bcero$. 
\item Iterative step until convergence. Suppose we have $\by_j$, $\bx_j$. 
\begin{enumerate}
\item Solve for the unconstrained optimization problem 
$$
\hbox{Minimize in }\bz\in\R^N: \quad f(\bz)+\sum_{k=1}^2e^{y_j^{(k)} g_k(\bz)}
$$
starting from the intial guess $\bx_j$. Let $\bx_{j+1}$ be such (local) minimizer.
\item If $\by_j\cdot \bbg(\bx_{j+1})$
vanishes, stop and take $\by_j$ as the multiplier of the problem, and $\bx_{j+1}$ as the solution of the constrained problem.
\item If $\by_j\cdot \bbg(\bx_{j+1})$ does not vanish, update 
$$
y_{j+1}^{(k)}=e^{y_j^{(k)} g_k(\bx_{j+1})}y_j^{(k)}
$$
for $k=1, 2$. 
\end{enumerate}
\end{enumerate}
Keep in mind the two functions $y_1(y_2)$, and $y_2(y_1)$, as defined in the discussion above: $y_1(y_2)$ corresponds to the multiplier for the one-dimensional situation (Section \ref{diuno}) with objective functional $f(\bx)+e^{y_2g_2(\bx)}$, and constraint $g_1(\bx)\le0$. Similarly for the other one. 

\begin{theorem}
Suppose the cost function $f:\R^N\to\R$, and the two components of the constraint map $\bbg:\R^N\to\R^2$ comply with:
\begin{enumerate}
\item they all are smooth, and convex;
\item the corresponding $L(\bx, \by)$ is coercive in $\bx$ for every fixed $\by$ with (strictly) positive components;
\item the mathematical problem \eqref{mathprog} is well-balanced.
\end{enumerate}
Then the above algorithm always converges to a solution of (\ref{optimalidaddos}) which is a (global) minimizer for (\ref{mathprog}).
\end{theorem}
\begin{proof}
Again, we may assume through a standard perturbation argument, and without loss of generality, that $f$ is strictly convex. In this way, the master function $L(\bx, \by)$ is coercive and strictly convex in $\bx$, for every $\by>\bcero$, and the mapping taking each $\by$ into the unique (global) minimizer $\bx\equiv\bx(\by)$ of $L(\bx, \by)$ with respect to $\bx$ is well-defined, and smooth. 

Let us consider the smooth mapping $\bG:\R^2_+\mapsto\R^2_+$ carrying $\by$ into 
$$
\left(y^{(k)}e^{y^{(k)}g_k(\bx(\by))}\right)_{k=1, 2} .
$$
Our statement is concerned with the fixed points of $\bG$. 

Because the problem is well-balanced, take $\tilde\by=(\tilde y_1, \tilde y_2)$ with components large enough, but definitely with $\tilde y_1>y_2(0)\ge0$, $\tilde y_2>y_1(0)\ge0$, where the functions $y_1$ and $y_2$ have been introduced above, and such that $\overline g(\tilde\by)<0$. Recall that $y_i(0)$ is the multiplier associated with the program
$$
\hbox{Minimize in }\bx\in\R^N:\quad f(\bx)\quad\hbox{ subject to }\quad g_i(\bx)\le0,
$$
for $i=1, 2$. 

We then claim that $\tilde y_1>y_1(\tilde y_2)$, and, similarly,  $\tilde y_2>y_2(\tilde y_1)$. Indeed, either $y_1(\tilde y_2)=0$, and our first inequality is correct; or else, by definition of $y_1(\tilde y_2)$, we should have that $\overline g_1(y_1(\tilde y_2), \tilde y_2)=0$, while $\overline g_1(\tilde y_1, \tilde y_2)<0$ by hypothesis. By the discussion with the single-inequality constraint case in the proof of Theorem \ref{onedimension} (right after the proof of Lemma \ref{auxx}, item (1)), this implies the claim $\tilde y_1>y_1(\tilde y_2)$. 
Likewise, for the other case.  

This conclusion, together with the previous choice of $\tilde\by$, immediately implies that the graphs of the two functions  intersect (at least) in a point $(\overline y_1, \overline y_2)\in\overline{\R^2_+}$ (the closure of $\R^2_+$). As indicated earlier in the discussion before the statement of the theorem, this intersection point generates a solution of (\ref{optimalidaddos}). We can even take $\tilde\by$, with larger components if necessary, so that $\tilde y_1>y_1(y_2)$ for all $y_2$ in the interval $[0, \tilde y_2]$, and $\tilde y_2>y_2(y_1)$ for all $y_1\in[0, \tilde y_1]$. 

\begin{center}
\includegraphics[width=8cm]{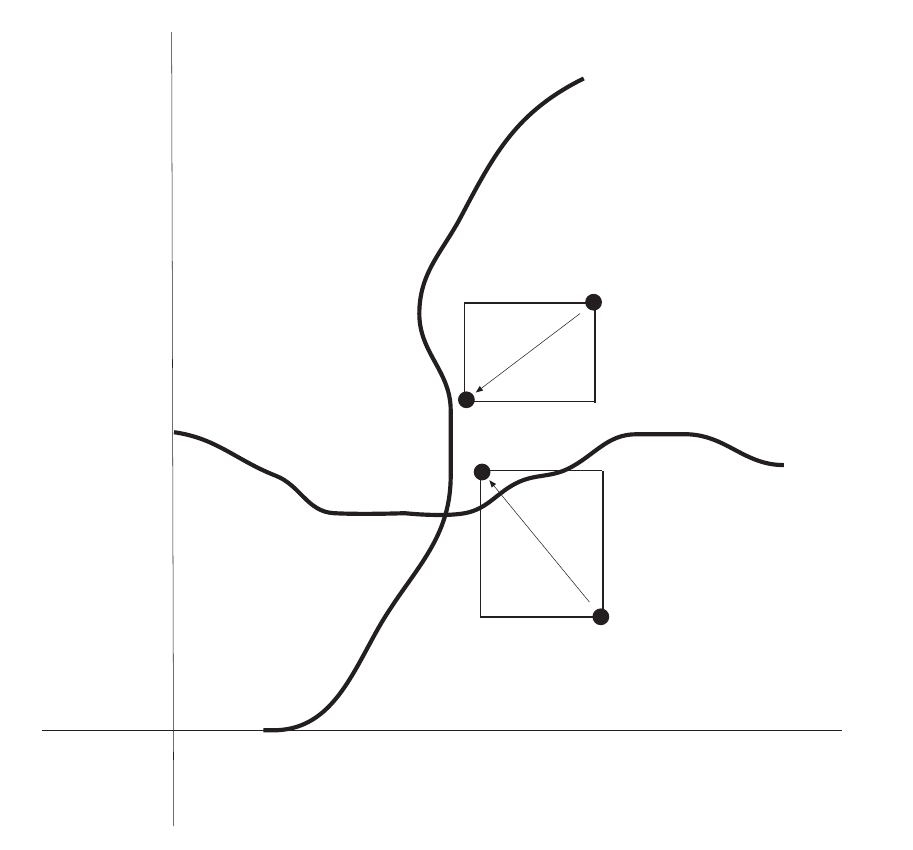}

Figure 2. The iteration rule for the two-dimensional situation.  
\vspace*{0.5cm} 
\end{center}

Let us now deal with the convergence issue. It is based on the global convergence for the one dimensional situation, and from this point of view is not difficult to show. Refer to Figure 2.

Because of  our final choice of the vector $\tilde\by$, it is clear that the whole sequence of iterates is uniformly bounded because each update rule 
\begin{equation}\label{actualizacion}
\overline y_k\mapsto z_k=\overline y_k e^{\overline y_kg_k(\overline\bx)},\quad k=1, 2,
\end{equation}
tends to the graph of the corresponding function $y_2(y_1)$, or $y_1(y_2)$, for both components, respectively, according to the discussion prior to the statement of the theorem. Indeed if we have a certain iterate $(\overline\by, \overline\bx)$ where $\overline\bx=\bx(\overline\by)$, 
again by the remarks made before the statement of the theorem, the difference $\overline y_1-z_1$ in (\ref{actualizacion}) corresponds to the one-dimensional process when the second variable $y_2$ is frozen at the value $\overline y_2$. Likewise for the difference $\overline y_2-z_2$ when $y_1=\overline y_1$. Because iterates stay in the fixed box $[0, \tilde y_1]\times[0, \tilde y_2]$, we conclude the claimed convergence through the one-dimensional situation. 
\end{proof}

We finally come to the general situation in which we become interested in the general mathematical program
\begin{equation}\label{probgen}
\hbox{Minimize in }\bx\in\R^N: \quad f(\bx)\quad\hbox{ subject to }\quad\bbg(\bx)\le\bcero,
\end{equation}
where $\bbg=(g_k):\R^N\to\R^m$. The fundamental map around which revolves our algorithm takes a vector of multipliers $\by=(y^{(k)})\in\R^m$, $\by>\bcero$, into a (local) solution $\bx(\by)$ of our basic unconstrained problem
\begin{equation}\label{unconsprob}
\hbox{Minimize in }\bx\in\R^N:\quad f(\bx)+\sum_{k=1}^m e^{y^{(k)}g_k(\bx)}.
\end{equation}
The objective function
$$
L(\bx, \by)=f(\bx)+\sum_{k=1}^m e^{y^{(k)}g_k(\bx)}
$$
of this unconstrained problem
is the master function of the problem. 

The algorithm has already been described:
\begin{enumerate}
\item Initialization. Take  $\by_0\in\R^m$, $\by_0>\bcero$, and $\bx_0\in\R^N$ in an arbitrary way. For instance, $\by_0=\buno$, $\bx_0=\bcero$. 
\item Iterative step until convergence. Suppose we have computed $\by_j$, and $\bx_j$. 
\begin{enumerate}
\item Solve  the unconstrained optimization problem 
$$
\hbox{Minimize in }\bz\in\R^N: \quad L(\bz, \by_j)=f(\bz)+\sum_{k=1}^me^{\by_j^{(k)} g_k(\bz)}
$$
starting from the intial guess $\bx_j$. Let $\bx_{j+1}$ be such (local) minimizer.
\item If $\by_j\cdot \bbg(\bx_{j+1})$
vanishes, stop,  take $\by_j$ as the multiplier of the problem, and $\bx_{j+1}$ as the solution of the constrained problem.
\item If $\by_j\cdot \bbg(\bx_{j+1})$ does not vanish, update 
$$
y_{j+1}^{(k)}=e^{y_j^{(k)} g_k(\bx_{j+1})}y_j^{(k)}
$$
for all $k=1, 2, \dots, m$. 
\end{enumerate}
\end{enumerate}

The proof of Theorem \ref{principalg} follows exactly the same strategy as with the two-component case.



\begin{thebibliography}{99}
\bibitem{bhatti} M. A. Bhatti, \textit{Practical Optimization Methods with Mathematica Applications}, Springer-Verlag, New York, 2000.

\bibitem{bertsekas} D. P. Bertsekas, \textit{Constrained Optimization and Lagrange Multiplier Methods}, Athena Sci., Bermont, Massachusetts, 1996. 

\bibitem{fiaccomccormick} Fiacco, A. V., McCormick, G. P., \textit{Nonlinear Programming. Sequential Unconstrained Minimization Techniques}, SIAM Classics in Appl. Math., 4, Philadelphia.

\bibitem{nocedal} Nocedal, J.,  Wright, S. J. \textit{Numerical optimization}, Springer Series in Operations Research, Springer-Verlag, New York, 1999. 

\bibitem{hiriarturruty}  Hiriart-Urruty, J. B., LemarŽchal, C., \textit{Convex analysis and minimization algorithms. I. Fundamentals},  Grundlehren der Mathematischen Wissenschaften [Fundamental Principles of Mathematical Sciences], 305. Springer-Verlag, Berlin, 1993.

\bibitem{hockschit} Hock, W., Schittkowski, K., \textit{Test examples for nonlinear programming codes}, Lect. Notes Econ. Math. Syst., 187, Springer, 1981.
\end{thebibliography}
\end{document}